\documentclass[11pt]{amsart}
\usepackage[margin=1.2in]{geometry}
\usepackage{amsmath}
\usepackage{color}
\usepackage{amsthm}
\usepackage{amsfonts}
\usepackage{amssymb}
\usepackage{amscd}
\numberwithin{equation}{section}
\newtheorem{theorem}{Theorem}[section]

\newtheorem{proposition}[theorem]{Proposition}
\newtheorem{lemma}[theorem]{Lemma}
\newtheorem{prop}[theorem]{Proposition}
\theoremstyle{definition}
\newtheorem{definition}[theorem]{Definition}
\newtheorem{remark}[theorem]{Remark}

\newcommand\q{\mathfrak q}

\newcommand\I{\Cal I}

\newcommand\be{\beta}

\newcommand\g{\mathfrak g}

\newcommand\h{\mathfrak h}

\newcommand\D{\Delta}
\renewcommand\l{\lambda}
\newcommand\Dp{\Delta^+}

\renewcommand\d{\delta}

\renewcommand\t{\mathfrak t}
\renewcommand\a{\alpha}
\renewcommand\aa{\mathfrak a}

\newcommand{\Z}{\mathbb Z}
\renewcommand\th{\theta}

\newcommand\nat{\mathbb N}
\newcommand\ganz{\mathbb Z}

\newcommand\s{\sigma}
\renewcommand\L{\Lambda}
\renewcommand\aa{\mathfrak a}

\newcommand\e{\epsilon}
\newcommand\C{\mathbb C}
\newcommand\R{\mathbb R}

\renewcommand\I{\sqrt{-1}}

\newcommand{\fg}{\mathfrak{g}}

\newcommand{\ZZ}{\mathbb{Z}}

\newcommand{\sdim}{\text{\rm sdim}}

\newcommand{\vac}{{\bf 1}}
\newcommand{\bea}{\begin{eqnarray}}
\newcommand{\eea}{\end{eqnarray}}
\begin{document}
\title{Unitarity of minimal $W$-algebras}
\author[Victor~G. Kac, Pierluigi M\"oseneder Frajria,  Paolo  Papi]{Victor~G. Kac\\Pierluigi M\"oseneder Frajria\\Paolo  Papi}
\begin{abstract} We obtain a complete classification of minimal simple unitary $W$-algebras.  \end{abstract}
\date{\today}
\maketitle
\section{Introduction}
In the present paper we study unitarity of minimal $W$-algebras. They are the simplest conformal vertex algebras among the simple vertex algebras
$W_k(\g,x,f),$ associated to a datum $(\g,x,f)$ and $k\in\R$. Here $\g=\g_{\bar 0}\oplus \g_{\bar 1}$ is a basic simple Lie superalgebra, i.e. its even part $\g_{\bar 0}$ is a reductive Lie algebra and $\g$ carries a non-zero even invariant non degenerate supersymmetric bilinear form $(\cdot|\cdot)$, $x$ is an $ad$-diagonalizable element of $\g_{\bar 0}$ with eigenvalues in $\tfrac{1}{2}\Z$, $f\in\g_{\bar 0}$ is such that $[x,f]=-f$ and the eigenvalues of $ad\,x$ on the centralizer $\g^f$ of $f$ in $\g$ are non positive, and  $k\ne -h^\vee$, where $h^\vee$ is the dual Coxeter number. The most important examples are provided by $x$ and $f$ to be part of an $sl_2$ triple $\{e,x,f\}$, where $[x,e]=e, [x,f]=-f, [e,f]=x$. In this case $(\g,x,f)$ is called a {\it Dynkin datum}.\par
We proved in \cite[Lemma 7.3]{KMP}  that if $\phi$ is a conjugate linear involution of $\g$ such that
\begin{equation}\label{prima} \phi(x)=x,\quad \phi(f)=f\text{ and }\overline{(\phi(a)|\phi(b))}=(a|b),\,a,b\in\g,\end{equation} then $\phi$ induces a conjugate linear involution of the vertex algebra $W_k(\g,x,f)$.\par

We also proved in \cite[Proposition 7.4]{KMP} that if $\phi$ is a conjugate linear involution of $W_k(\g,x,f)$, this vertex algebra   carries a non-zero $\phi$-invariant Hermitian form $H(\cdot,\cdot)$ for all $k\ne - h^\vee$ if and only if  $(\g,x,f)$ is  a Dynkin datum; moreover, such $H$  is unique, up to a real constant factor, and we normalize it by the condition $H(\vac,\vac)=1$. The vertex algebra is called {\it unitary} if  there is a conjugate linear involution $\phi$ such that the corresponding $\phi$--invariant Hermitian form $H$ is positive definite.\par
For some levels $k$ the vertex algebra $W_k(\g,x,f)$ is trivial, i.e. isomorphic to $\C$; then it is trivially unitary. Another easy case is when  $W_k(\g,x,f)$ ``collapses'' to the affine part.
In both cases we will say that $k$ is {\it collapsing level}.\par
Let $\g^\natural$ be the centralizer of the $sl_2$ subalgebra $\mathfrak s=span\, \{e,x,f\}$ in $\g_{\bar 0}$; it is a reductive subalgebra. If $\phi$ satisfies the first two conditions in \eqref{prima}, it fixes 
$e,x,f$, hence $\phi(\g^\natural)=\g^\natural$. It is easy to see that unitarity of $W_k(\g,x,f)$ implies, when $k$ is not collapsing,  that $\phi_{|\g^\natural}$ is a compact involution of the reductive Lie algebra $\g^\natural$.\par
In the present paper we consider only {\it minimal} data $(\g,x,f)$, defined by the property that for the $ad\,x$ gradation $\g=\bigoplus\limits_{j\in\tfrac{1}{2}\mathbb Z}\g_j$ one has 
\begin{equation}\label{seconda}\g_j=0\text{ \ if $|j|>1$, and } \g_{-1}=\C f.\end{equation}
In this case $(\g,x,f)$ is automatically a Dynkin datum. The corresponding $W$-algebra is called {\it minimal}.
The element $f\in\g$ is a root vector attached to a root $\theta$ of $\g$, and we shall normalize the invariant bilinear form on $\g$ by the condition 
$(\theta|\theta)=2$.
Recall that the dual Coxeter number $h^\vee$ of $\g$  is half of the eigenvalue of its Casimir element of $\g$, attached to the bilinear form $(\cdot|\cdot)$.
We shall denote by   $W_k^{\min}(\g)$ the minimal $W$-algebra, corresponding to $\g$ and $k\ne -h^\vee$.
\par
We proved in \cite[Proposition 7.9]{KMP} that, if  $W^{\min}_k(\g)$ is non-trivial unitary and $k$ is not a collapsing level, then the parity of $\g$ is compatible with the $ad\,x$-gradation i.e. the parity of the whole subspace $\g_j$ is $2j\mod 2$.

\par
It follows from \cite{KRW}, \cite{KW1} that for each basic simple Lie superalgebra $\g$ there is at most one minimal Dynkin datum, compatible with parity, and  the  complete list of the  $\g$ which admit such a datum is as follows:
\begin{align}\label{ssuper}
\begin{split}
sl(2|m)\text{ for $m\ge3$},\quad psl(2|2),\quad  spo(2|m)\text{ for $m\geq 0$,}\\
osp(4|m)\text{ for $m>2$ even},\quad D(2,1;a) \text{ for }a\in \C,\quad  F(4),\quad G(3).
\end{split}
\end{align}
The even part $\g_{\bar 0}$ of $\g$ in this case is isomorphic to the direct sum of a reductive Lie algebra $\g^\natural$ and $\mathfrak s\cong sl_2$. As has been explained above, a minimal $W$-algebra $W_k^{\min}(\g)$ with non-collapsing $k$ can be unitary only if 
 the conjugate linear involution $\phi$ on $\g$ is {\it almost compact}, 
according to the following definition.
\begin{definition}\label{ac} We say that a  conjugate linear involution $\phi$ on $\g$ is  almost compact if
\begin{enumerate}
\item[(i)] $\phi$ fixes $e,x,f$;
\item[(ii)] $\phi$ is a compact conjugate linear involution of $\g^\natural$.
\end{enumerate}
\end{definition}
Indeed (i) is equivalent to the first two requirements in \eqref{prima}, and the third requirement follows from Lemma \ref{31} in Section 3.

We prove that an almost compact conjugate linear involution $\phi$ exists for all $\g$ from the list \eqref{ssuper}, except that $a\in \R$, and is 
essentially unique. So by unitarity of a minimal $W$-algebra we mean unitarity for $\phi$ almost compact.
\par
It was shown in \cite{KW1} that the central charge of $W_k^{\min}(\g)$ equals
\begin{equation}\label{cc} c(k)=\frac{k\,d}{k+h^\vee}-6k+h^\vee-4,\text{ where $d=\sdim\g$.}\end{equation}
Here is another useful way to write this formula:
\begin{equation}\label{ccc} c(k)=7h^\vee+d-4-12 \sqrt{\phantom{x}} - 6\frac{(k+h^\vee-\sqrt{\phantom{x}})^2}{k+h^\vee},\text{ where $\sqrt{\phantom{x}}=\sqrt{\frac{d\,h^\vee}{6}}$.}\end{equation}

Recall that all well-known superconformal algebras in conformal field theory are the mini\-mal $W$-algebras or are obtained from them by a simple modification:
\begin{enumerate}
\item[(a)] $W^{\min}_k(spo(2|N))$ is the Virasoro vertex algebra for $N=0$, the Neveu-Schwarz vertex algebra for $N=1$, the $N=2$ vertex algebra for $N=2$, and becomes the $N=3$ vertex algebra 
after tensoring with one fermion; it is the Knizhnik algebra for $N>3$;
\item[(b)]  $W^{\min}_k(psl(2|2))$ is the $N=4$ vertex algebra;
\item[(c)]  $W^{\min}_k(D(2,1;a))$ tensored with four fermions and one boson is the big $N=4$ vertex algebra;
\end{enumerate}\par
The unitary Virasoro ($N=0$), Neveu-Schwarz ($N=1$) and  $N=2$ simple vertex algebras were classified in the mid 80s.  Up to isomorphism,  these vertex algebras depend only on the central charge $c(k)$, given by  \eqref{cc}. Putting $k=\frac{1}{p}-1$ in \eqref{ccc} in all three cases, we obtain 
\begin{align}
\label{1}c(k)&=1-\frac{6}{p(p+1)} \quad \text{for   Virasoro vertex algebra,}\\
\label{22}c(k)&=\frac{3}{2}\left(1-\frac{8}{p(p+2)}\right)\quad \text{for   Neveu-Schwarz vertex algebra,}\\
\label{3}c(k)&=3\left(1-\frac{2}{p}\right)\quad \text{for   $N=2$ vertex algebra.}
\end{align}
The following theorem is a result of several papers, published in the 80s in physics and mathematics literature, see e.g. \cite{ET3} for references.
\begin{theorem}The complete list of unitary $N=0,1,$ and $2$ vertex algebras is as follows: either $c(k)$ is given by \eqref{1}, \eqref{22}, or \eqref{3} respectively for $p\in \mathbb Z_{\ge 2}$  or  $c(k)\ge 1, \frac{3}{2}, 3$ respectively. 
\end{theorem}
The above three cases cover all minimal $W$-algebras, associated with $\g$, such that the $0$th eigenspace $\g_0$ of $ad\,x$ is abelian. Thus, we may assume that $\g_0$ is not abelian. Under this assumption we obtain a complete classification of minimal simple unitary $W$-algebras (cf. Theorem 	\ref{fin}).
\begin{theorem} The simple minimal $W$-algebra $W_{-k}^{\min}(\g)$ with $k\ne h^\vee$ and $\g_0$ non-abelian is non-trivial unitary if and only if 
\begin{enumerate}
\item $\g=sl(2|m), m>3$, $k=1$;
\item  $\g=psl(2|2)$, $k\in\mathbb N +1$;
\item  $\g=spo(2|3)$, $k\in\frac{1}{4}(\mathbb N+2)$;
\item $\g=spo(2|m)$, $m>3$, $k=\frac{1}{2}(\mathbb N+1)$;
\item   $\g=D(2,1;-\frac{m}{m+n})$, $k=\frac{mn}{m+n}, \text{ where }m,n\in\mathbb N,\ m+n>2$;
\item   $\g=F(4)$, $k\in\frac{2}{3}(\mathbb N+1)$;
\item  $\g=G(3)$, $k\in\frac{3}{4}(\mathbb N+1)$.
\end{enumerate}
\end{theorem}
The case (1) is just the free boson vertex algebra. The results (2) and  (3) of Theorem \ref{12} are consistent with the results 
of \cite{ET1} and \cite{M} respectively.\par
By Theorem \ref{main} (b) the vertex algebra $W_k^{\min}(\g)$, with $k\neq -h^\vee$ and 
non-collapsing, has a
non-trivial unitary module only when $W_k^{\min}(\g)$ is a unitary vertex 
algebra. We are planning to classify its unitary modules and compute their 
characters in a subsequent publication.\par
Throughout the paper the base field is $\C$, and $\mathbb Z_+$ and $\mathbb N$ stand  for the set of non negative and positive integers, respectively.\par
{\bf Acknowledgements}.  Victor Kac is partially supported by the Bert and Ann Kostant fund and
the Simons collaboration grant.

\section{Setup}
\subsection{Basic simple Lie superalgebras} Let  $\g=\g_{\bar 0}\oplus \g_{\bar 1}$ be a basic simple  finite-dimensional Lie superalgebra over $\C$ as in 
\eqref{ssuper}.
Choose a Cartan subalgebra  $\h$ of $\g_{\bar 0}$. It is a maximal $ad$-diagonalizable subalgebra of $\g$, for which the root space decomposition is of the form
\begin{equation}\label{rs}
\g=\h\oplus\bigoplus_{\a\in\D}\g_\a,
\end{equation}
where $\D\subset\h^*\setminus\{0\}$ is the set of roots. In all cases, except for $\g\cong psl(2|2)$,
the root spaces have dimension $1$. In the case $\g=psl(2|2)$ one can achieve this property by embedding in $pgl(2|2)$ and replacing 
\eqref{rs} by the root space decomposition with respect to a Cartan subalgebra of $pgl(2|2)$, which we will do.\par
Let $\Dp$ be a subset of positive roots and $\Pi=\{\a_1,\ldots,\a_r\}$ be the corresponding set of simple roots. Set $\Pi_1$ be the set of simple odd roots.
For each $\a\in\Dp$ choose $X_\a\in \g_\a$ and $X_{-\a}\in\g_{-\a}$ such that $(X_\a|X_{-\a})=1$,
 and let $h_\a=[X_\a,X_{-\a}]$. Let  $e_i=X_{\a_i}, f_i=X_{-\a_i},\,i=1,\ldots,r$.  The set $\{e_i, f_i, h_{\a_i}\mid i=1,\ldots,r\}$ 
generates $\g$, and satisfies the following  relations 
\begin{equation} \label{r1}
[e_i,f_j]=\d_{ij}h_{\a_i},\quad
[h_{\a_i},e_j]=(\a_i|\a_j)e_j,\quad
[h_{\a_i},f_j]=-(\a_i|\a_j)f_j.
\end{equation}
The Lie superalgebra $\tilde \g$ on generators  $\{e_i, f_i, h_{\a_i}\mid i=1,\ldots,r\}$ subject to relations \eqref{r1} is a (infinite-dimensional) $\Z$-graded Lie algebra, where the grading is defined by $\deg h_{\a_i}=0, \deg e_i=-\deg f_i=1$,   with a unique $\Z$-graded  maximal ideal, and $\g$ is the quotient of $\tilde \g$ by this ideal.
Note that 
$(\a_i|\a_j)\in\R$ for any $\a_i,\a_j\in\Pi$.\par
\subsection{Conjugate linear involutions and real forms}\label{222} In the above setting, given a collection of  complex numbers $\L=\{\l_1,\ldots,\l_r\}$ such that $\l_i\in \sqrt{-1}\R$ if $\a_i$ is an odd root and $\l_i\in \R$ if $\a_i$ is an even root,
we can define an antilinear involution $\omega_\L:\g\to\g$ setting
\begin{equation}\label{omega}
\omega_\L(e_i)=\l_if_i,\quad\omega_\L(f_i)={\bar\l_i}^{-1}e_i,\quad
\omega_\L(h_{\a_i})=-h_{\a_i},\ 1\leq i\leq r.
\end{equation}
Since $\omega_\L$ preserves  relations \eqref{r1}, it induces  an antilinear involution of $\tilde\g$, and, since $\omega_\L$ preserves the $\Z$-grading of $\tilde\g$, it preserves its unique maximal ideal, hence it induces an antilinear involution of $\g$.

Set $\sigma_\a=-1$ if $\a$ is an odd negative root and  $\sigma_\a=1$ otherwise, so that 
$(X_\a|X_{-\a})=\sigma_\a$. Let 
$$\xi_\a=\begin{cases}sgn(\a|\a) \quad&\text{if $\a$ is an even root,} \\ 1\quad &\text{if $\a$ is an odd root.}\end{cases}$$
Then in \cite[(4.13), (4.15)]{GKMP} it is proven (using results from \cite{YK}), that one can choose  root vectors $X_\a$ in such a way that 
\begin{equation}\label{ox}
\omega_{\L}(X_\a)=-\s_\a\xi_\a\l_\a X_{-\a},
\end{equation}
where 
\begin{equation} \label{lambdaalfa}
\l_\a=\prod_i(-\xi_{\a_i}\l_i)^{n_i}\text{ \ for $\a=\sum_{i=1}^{r}n_i\a_i$.}
\end{equation} 
We shall call this a {\it good choice} of root vectors.

\subsection{Invariant Hermitian forms on vertex algebras} Let $V$ be a conformal vertex algebra with conformal vector $L=\sum_{n\in\mathbb Z} L_nz^{-n-2}$ (see \cite{KMP} for the definition and undefined notation).   Let $\phi$ be a conjugate linear involution of $V$. A Hermitian form $H(\, \cdot \,\, , \, \cdot\, )$ on $V$ is called $\phi$--{\it invariant} if, for all $a\in V$, one has \cite{KMP}
\begin{equation}\label{111}
H(v,Y(a,z)u)=H(Y(A(z)a,z^{-1})v,u),\quad u,v\in V.
\end{equation}
Here the linear map $A(z):V\to V((z))$ is defined by 
\begin{equation}\label{12}
A(z)=e^{zL_1}z^{-2L_0}g,
\end{equation}
where
\begin{equation}\label{113}
g(a)=e^{-\pi\sqrt{-1}(\tfrac{1}{2}p(a)+\D_a)}\phi(a),\quad a\in V,
\end{equation}
$\D_a$ stands for the   $L_0$--eigenvalue of $a$, and 
$$p(a)=\begin{cases}0\in\Z&\text{ if $a\in \g_{\bar{0}}$},\\ 1 \in\Z&\text{ if $a\in \g_{\bar{1}}$. }
\end{cases}
$$

\begin{definition}\label{unidef} We say that a conformal vertex algebra $V$ is {\sl unitary} if there exists a conjugate linear involution $\phi$ of $V$ and a 
$\phi$-invariant positive definite Hermitian form on $V$.
\end{definition}
\section{The almost compact conjugate linear involution of $\g$}\label{2}
 From now on we let $\g$ be a basic simple finite-dimensional Lie superalgebra such that 
 \begin{equation}\label{goss}
\g_{\bar 0}=\mathfrak s\oplus \g^\natural.
\end{equation}
where  $\mathfrak s\cong sl_2$ and $\g^\natural$ is the centralizer of $\mathfrak s $ in $\g$.

This corresponds to consider $\g$  as in  Table 2 of \cite{KW1}. We will also assume that $\g^\natural$ is not abelian; this condition rules out   $\g=spo(2|m),\ m=0,1,2$. The explicit list is given in the leftmost column of Table 1. Note that $sl(2|1)$ and $osp(4|2)$ are missing there since $sl(2|1)\cong spo(2|2)$ and
$osp(4|2)\cong D(2,1;-\tfrac{1}{2})$.
\par
First, we prove the simple lemma mentioned in the Introduction, which implies the first two conditions of \eqref{prima} imply the third one.
\begin{lemma}\label{31} Let $\g$ be a simple Lie superalgebra with an invariant supersymmetric bilinear form $(\cdot|\cdot)$, let $x\in \g$, and let $\phi$ be a conjugate linear involution of $\g$, such that 
\begin{equation}\label{aux} 
(x|x)\text{ is a non-zero real number, and } \phi(x)=x.
\end{equation} Then 
\begin{equation}\label{forma}\overline{(\phi(a)|\phi(b))}=(a|b),\text{ for all  $a,b\in\g$.}\end{equation}
\end{lemma}
\begin{proof} Note that $\overline{(\phi(a)|\phi(b))}$ is an  invariant supersymmetric bilinear form as well, hence it is proportional to $(a|b)$ since $\g$ is simple. Due to \eqref{aux}
these two bilinear forms coincide.
\end{proof}
\vskip10pt
\begin{prop} \label{constructphi} For any $sl_2$-triple $\{e,x,f\}$, such that \eqref{goss} holds for $\mathfrak s=span \{e,x,f\}$, an almost compact involution exists.
\end{prop}
\begin{proof}  Choose  a Cartan subalgebra $\mathfrak t$ of $\g_{\bar 0}$. We observe that if we prove the existence of an almost compact involution $\phi$ for
a special choice of $\{e,x,f\}$, then an almost compact involution exists for any choice of the $sl_2$-triple. Indeed, if $\{e',x',f'\}$ is another 
$sl_2$-triple, then there is an inner automorphism $\psi$ of $\mathfrak s$ mapping $\{e,x,f\}$ to $\{e',x',f'\}$, which extends to an inner automorphism of $\g$. Therefore 
$\phi'=\psi\phi\psi^{-1}$ is an almost compact involution for $\{e',x',f'\}$.
The costruction of  $\{e,x,f\}$ and $\phi$ and the verification of properties (i)-(iii) in Definition \ref{ac} will be done in four steps:
\begin{enumerate}
\item make a suitable choice of positive roots for $\g$ with respect to  $\mathfrak t$; 
\item define $\phi$ by specializing \eqref{omega};
\item construct $\{e,f,x\}$ and verify that  $\phi(f)=f, \phi(x)=x$, $\phi(e)=e$;
\item check  that  $\phi$ is a compact involution for  $\g^\natural$;
\end{enumerate}
{\sl Step 1.} We need some preparation. Let 
$\D^\natural$ be the set of roots  of  $\g^\natural$ with respect to the Cartan subalgebra $\mathfrak t\cap \g^\natural$.
Let $\{\pm\theta\}$ be the $\mathfrak t\cap \mathfrak s$-roots of $\mathfrak s$. Then $R_{\bar 0}=\{\pm\theta\}\cup\D^\natural$ is the set of roots of $\g_{\bar 0}$ with respect to $\t$. 

Let $R$ be the set of roots of $\g$ with respect to $\t$, let  $R^+$ be subset of positive roots whose corresponding set of simple roots $S=\{\a_1,\ldots,\a_r\}$ is displayed in Table 1. 
\begin{table}
{\scriptsize
\begin{tabular}{c | c| c |c }
$\g$&$S$&\text{ $(\cdot |\cdot)$} & $\theta$\\
\hline
$psl(2|2)$&
$\{\e_1-\d_1,\d_1-\d_2,\d_2-\e_2\}$&$(\e_i|\e_j)=\d_{i,j}=-(\d_i|\d_j)$ & $\e_1-\e_2$\\
& &$(\e_i|\d_j)=0$
\\\hline
$sl(2|m), m>2$&
$\{\e_1-\d_1,\d_1-\d_2,\ldots,\d_m-\e_2\}$&$(\e_i|\e_j)=\d_{i,j}=-(\d_i|\d_j)$ & $\e_1-\e_2$\\
& &$(\e_i|\d_j)=0$
\\\hline
$osp(4|m), m>2$&$\{\e_1-\e_2,\e_2-\d_1,\d_1-\d_2,\ldots,\d_{m-1}-\d_m,2\d_m\}$&$(\e_i|\e_j)=\d_{i,j}=-(\d_i|\d_j)$ & $\e_1+\e_2$\\
& &$(\e_i|\d_j)=0$\\\hline
$spo(2|2m+1), m\ge1$&
$\{\d-\e_1,\e_1-\e_2,\ldots,\e_{m-1}-\e_m,\e_m\}$&$(\e_i|\e_j)=-\tfrac{1}{2}\d_{i,j}, (\d|\d)=\tfrac{1}{2}$& $2\d$
\\\hline
$spo(2|2m), m\ge 3$&
$\{\d-\e_1,\e_1-\e_2,\ldots,\e_{m-1}-\e_m,\e_{m-1}+\e_m\}$&$(\e_i|\e_j)=-\tfrac{1}{2}\d_{i,j}, (\d|\d)=\tfrac{1}{2}$& $2\d$
\\\hline
$D(2,1;a)$&$\{\e_1-\e_2, \e_2-\delta, 2\delta\}$& $(\e_i|\e_j)=\delta_{i,j}, (\delta|\delta)=-(1+a)/2$ & $\e_1+\e_2$\\
 & & $(\e_1|\delta)=(\e_2|\delta)=(1-a)/4$\\\hline
$F(4)$&$\{\tfrac{1}{2}(\d-\e_1-\e_2-\e_3), \e_3,\e_2-\e_3,\e_1-\e_2\}$&$(\e_i|\e_j)=-\tfrac{2}{3}\d_{i,j}, (\d|\d)=2$& $\d$\\
& &$(\e_i|\d)=0$
\\\hline
$G(3)$&$\{\d+\e_3, \e_1,\e_2-\e_1\}$&$(\e_i|\e_j)=\frac{1-3\d_{i,j}}{4}, (\d|\d)=\frac{1}{2}$& $2\d$\\
& &$(\e_i|\d)=0,\ \e_1+\e_2+\e_3=0$
\end{tabular}
}

\caption{}
\end{table}
Note that   $\theta$ is the  highest root of 	$R$.
\vskip5pt
\noindent{\sl Step 2.} Define 
\begin{equation}\L_0=\{\l_1,\ldots,\l_r\},\quad\l_i=\begin{cases}- sgn (\a_i|\a_i)\quad&\text{if $\a_i$ is  even,}\\
\sqrt{-1}\quad&\text{if $\a_i$ is  odd.}\end{cases}\label{lambda0}\end{equation}
 Set 
  $\phi=\omega_{\L_0}$ (see \eqref{omega}).

 \vskip5pt
\noindent{\sl Step 3.} Consider a good choice of root vectors $X_\a$ for $\L_0$.
Set 
\begin{equation}\label{efx} x=\tfrac{\sqrt{-1}}{2}(X_\theta-X_{-\theta}), \quad 
e=\tfrac{1}{2}(X_\theta+X_{-\theta}+\sqrt{-1}h_\theta), \quad f=\tfrac{1}{2}(X_\theta+X_{-\theta}-\sqrt{-1}h_\theta).\end{equation}
If  $\theta=\sum_{i=1}^{r}m_i\a_i$, then, by our special choice of $\Dp$, we have  either $m_i=2$ for exactly one odd simple root $\a_i$, or $m_i=m_j=1$ for exactly two odd distinct simple roots $\a_i,\a_j$ (this corresponds to the fact that $R^+$ is distinguished, in the terminology of \cite{GKMP}). By \eqref{ox} we have
\begin{equation}\label{phixtheta}\phi(X_\theta)=- (\sqrt-1)^2 X_{-\theta}=X_{-\theta}. 
\end{equation}
Since $h_\theta=\sum_{i=1}^r m_ih_{\a_i}$ and $\phi(h_{\a_i})=-h_{\a_i}$, it is clear from \eqref{efx}
 that $\phi$ fixes $e, f, x$. One checks directly that $\{e,f,x\}$ is an $sl_2$-triple.

\vskip5pt
\noindent{\sl Step 4.} Endow $\g$ with the $\ganz$-grading
\begin{equation}\label{grading}\g=\bigoplus_{i\in\ganz} \mathfrak q_i\end{equation}
which assigns   degree $0$ to $h\in\t$ and to $e_i$ and $f_i$ if $\a_i$ is even,
and degree $1$ to $e_i$ and degree $-1$ to $f_i$, if $\a_i$ is odd.

A direct check on Table 1 shows that $\mathfrak q_0=\g^\natural$.
Recall from \cite[Proposition 4.5]{GKMP} that the fixed points of $\phi$ in $\mathfrak q_0$ are a compact form of $\q_0$ if and only  if 
$\l_i(\a_i|\a_i)<0$ for all $\a_i\in S\setminus S_1$. Step 4 now follows from \eqref{lambda0}. \par
\end{proof}
\section{Explicit expressions for almost compact forms}
 In this section we exhibit explicitly  an almost compact involution 
$\phi$ in each case and discuss its uniqueness. If $\phi$ is an almost compact involution of $\g$,  we denote by $\g^{ac}$ the corresponding real form  (the fixed point set of $\phi$).
We can define $\g^{ac}$ by specifying a real form $\g^{ac}_{\bar 0}$ of $\g_{\bar 0}$ and a real form $\g^{ac}_{\bar 1}$ of  $\g_{\bar 1}$.

(1)  $\g=spo(2|m)$. Then $\g_{\bar 0}=sl_2\oplus so_m$ and $\g_{\bar 1}=\C^2\otimes \C^m$ as $\g_{\bar 0}$-module. We set 
$$\g^{ac}_{\bar 0}=sl_2(\mathbb R)\oplus so_m(\mathbb R),\quad \g^{ac}_{\bar 1}=\R^2\otimes \R^m.
$$
Explicitly, let $B$ be a non-degenerate $\R$-valued bilinear form of the superspace $\R^{2|m}$ with matrix
$\left(\begin{array}{c c|c} 0 &1 &0\\-1 & 0& 0\\\hline  0& 0& I_m\end{array}\right)$. Then for $\g=spo(2|m)$ we have:
$$\g^{ac}=\{A \in sl(m|n;\R)\mid B(Au,v)+(-1)^{p(A)p(u)}B(u,Av)=0\}.
$$

(2) $\g=psl(2|2)$. 
Let $H$ be a $\C$-valued non-degenerate sesquilinear form on the superspace $\C^{2|2}$ whose matrix is $diag(\I,-\I,1,1)$. Set
$$\tilde\g^{ac}=\{A \in sl(2|2;\C)\mid H(Au,v)+(-1)^{p(A)p(u)}H(u,Av)=0\}.
$$
Then 
$$\g^{ac}=\tilde\g^{ac}/\R \sqrt{-1} I.$$
Explicitly, we have 
 $\g_{\bar 0}=sl_2\oplus sl_2$ and $\g_{\bar 1}=\left\{\left(\begin{array}{c|c} 0 & B \\\hline C & 0\end{array}\right)\mid B,C\in M_{2,2}(\C)\right\}$ as a $\g_{\bar 0}$-module. Then 
\begin{align*}\tilde \g^{ac}_{\bar 0}&=\left\{\left(\begin{array}{c|c} A & 0 \\\hline 0& D\end{array}\right)\mid  A\in su(1,1), D\in su_2\right\},\\
\tilde \g^{ac}_{\bar 1}&=\left\{\left(\begin{array}{c c |c}0 & 0 & u \\ 0 & 0  & v\\\hline
\I \bar u^t &-\I \bar v^t & 0\end{array}\right)\mid  u, v\in \C^2\right\}.
\end{align*}

(3)  $\g=D(2,1;a)$. Then $\g_{\bar 0}=sl_2\oplus sl_2\oplus sl_2= so(4,\C)\oplus sl_2$ and $\g_{\bar 1}=\C^2\otimes \C^2\otimes \C^2=
\C^4\otimes \C^2$ as $\g_{\bar 0}$-module. We set 
$$\g^{ac}_{\bar 0}=so(4,\R)\oplus span_\R\{e,f,x\},\quad \g^{ac}_{\bar 1}=\R^4\otimes \R^2.
$$
To get an explicit realization, consider
the contact Lie superalgebra (see \cite{Kacsuper} for more details)  $$K(1, 4)=\C[t,\xi_1,\xi_2,\xi_3,\xi_4]$$ where $t$ is an even variable and $\xi_i, 1\leq i\leq 4$, are  odd variables. Introduce on the associative superalgebra $K(1, 4)$ a $\Z$-grading by letting 
$$\deg' t=2,\quad \deg'\xi_i=1,$$
and the bracket
$$\{f,g\}=(2-\sum_{i=1}^4\xi_i\partial_i)f\partial_t g-\partial_t f(2-\sum_{i=1}^4\xi_i\partial_i)g+\sum_{i=1}^4
(-1)^{p(f)}\partial _if\partial_ig,
$$
where $\partial_i=\partial_{\xi_i}$. This is a $\Z$-graded Lie superalgebra with compatible grading 
$\deg  f=\deg'f-2$. We have
$$K(1,4)=\bigoplus_{j\geq -2} K(1,4)_j,$$
where
\begin{align*}&K(1,4)_{-2}=\C 1,  &&K(1,4)_{-1}=span_\C(\xi_i\mid 1\leq i \leq 4), \\&
K(1,4)_0=span_\C(\xi_i\xi_j, t\mid 1\leq i,j \leq 4),
&&K(1,4)_1=\g_1'\oplus \g_1'',\text{ where}\\&\g_1'=span_\C(t\xi_i\mid 1\leq i\leq 4), &&\g_1''=span_\C(\xi_i\xi_j\xi_k\mid 1\leq i,j,k \leq 4).
\end{align*}
Note that $span_\C(\xi_i\xi_j\mid 1\leq i,j \leq 4)=\Lambda^2\C^4\cong so(4,\C)$,  that $\g'_1$ is isomorphic to the standard representation $\C^4$ of $so(4,\C)$ and that $\g''_1$ is isomorphic to
$\Lambda^3\C^4$, so that $K(1,4)_1=\C^4\oplus \C^4$ as $so(4,\C)$-module. Also notice that $\{\g_1',\g_1'\}=\C t^2, \{\g_1'',\g_1''\}=0$.
Fix now a copy $\tilde \g_b$ of  an $so(4,\C)$-module $\C^4$  in $\C^4\oplus \C^4$, depending on a constant $b\in \R$,  as follows. Set, for $1\leq i\leq 4$,
$$a_i=t\xi_i+ b \hat \xi_i,\text{ where }  \hat \xi_i=(-1)^{i+1}\prod_{j\ne i} \xi_j,$$
and define
$$\tilde \g_b=\sum_{i=1}^4\C a_i.$$
 Let $b\in \R$. Note that, setting $\xi=\xi_1\xi_2\xi_3\xi_4$, we have 
\begin{align*}
\{t\xi_i+ b \hat \xi_i,t\xi_j+ b \hat \xi_j\}&=\d_{ij}(-t^2+2b\xi).
\end{align*}
Hence, if we set
$$e=-t^2+2b\xi,\quad f=-1,\quad x=t/2,$$
then $\{e,x,f\}$ is an $sl_2$-triple. Set 
$$\g^{ac}=\R\cdot 1\oplus \left(\sum_{i=1}^4\R \xi_i\right)\oplus\left( \sum_{i,j=1}^4\R \xi_i\xi_j\oplus \R\tfrac{t}{2}\right)\oplus\left( \sum_{i=1}^4\R a_i\right)\oplus \R(-t^2+2b\xi). $$
Then $\g^{ac}$ is an almost compact form  of  $D(2,1;\frac{1+b}{1-b})$. To prove this, it suffices to calculate the  Cartan matrix for a choice of Chevalley generators of the complexification of $\g^{ac}$.
Fix a Cartan subalgebra in $\g^\natural = so(4,\C)$ as the span of  $v_2=-\I\xi_1\xi_2, v_3=-\I\xi_3\xi_4$. Set  $v_1=t$; then $\{v_1,v_2,v_3\}$ is a basis of a  Cartan subalgebra of $\g$. Let $\{\e_1,\e_2,\e_3\}$ the   dual basis  to $\{v_1,v_2,v_3\}$. One can choose $\{\a_1=\e_2-\e_1,\a_2=\e_1-\e_3,\a_3=\e_1+\e_3\}$ as a set of simple roots. The associated Chevalley generators are 
{\small
\begin{align*}
 &e_1=-\I a_1+a_2 &&\!\!e_2=\xi_1\xi_3+\xi_2\xi_4+\I(\xi_1\xi_4-\xi_2\xi_3)&&e_3=\xi_1\xi_3-\xi_2\xi_4-\I(\xi_1\xi_4+\xi_2\xi_3)\\
 &f_1=\I \xi_1+\xi_2  &&\!\!f_2=\xi_1\xi_3+\xi_2\xi_4-\I(\xi_1\xi_4-\xi_2\xi_3)&&f_3=\xi_1\xi_3-\xi_2\xi_4+\I(\xi_1\xi_4+\xi_2\xi_3)\\
 &h_1=-2v_1+2v_2+2b\, v_3 &&\!\!h_2=4v_1-4v_3 &&h_3=4v_1+4v_3
\end{align*}
}
and the corresponding Cartan matrix, normalized as in \cite{Kacsuper},
is 
$\begin{pmatrix} 0 &1 &\frac{1+b}{1-b}\\ -1&2&0\\-1 & 0 &2\end{pmatrix}$. Hence $a=\frac{1+b}{1-b}$ and therefore all $a\neq -1$ occur in this 
construction. Since this subalgebra is $17$-dimensional, it is isomorphic to $D(2,1;a)$.
\begin{remark} 
Note that $a=0$ for $b=-1$. In this case, $D(2,1;0)$ contains a $11$-dimensional   solvable  ideal generated by $f_1$, which is spanned by
$h_1$ and the root vectors relative to roots having $\a_1$ in their support.
If  we replace $a_i$ by $a_i/b$ and $h_1$ by $h_1/b$, and tend $b$ to  $+\infty$, we  recover also the Lie superalgebra of derivations of $psl(2|2)$, and its almost compact real form.
\end{remark}
(4)  $\g=G(3)$. Then $\g_{\bar 0}=sl_2\oplus G_2$ and $\g_{\bar 1}=\C^2\otimes L_{\min}$, where  $L_{\min}$ is the 7-dimensional over $\C$ irreducible representation  of $G_2$, and we let 
$$\g^{ac}_{\bar 0}=sl_2(\mathbb R)\oplus G_{2,0},\quad \g^{ac}_{\bar 1}=\R^2\otimes  L_{\min,0}.
$$
where $G_{2,0}$ is the real compact form of $G_2$ and $L_{\min,0}$ is the 7-dimensional  over $\R$ irreducible  representation  of $G_{2,0}$

(5)  $\g=F(4)$. Then $\g_{\bar 0}=sl_2\oplus so_7$ and $\g_{\bar 1}=\C^2\otimes spin_7$, and we let
$$\g^{ac}_{\bar 0}=sl_2(\mathbb R)\oplus so_7(\R),\quad \g^{ac}_{\bar 1}=\R^2\otimes  spin (\R^7).
$$

It is proved in \cite[Proposition 5.3.2]{Kacsuper} that in both cases (4) and (5) $\g^{ac}=\g^{ac}_{\bar 0}\oplus \g^{ac}_{\bar 1}$ is an almost compact form of $\g$.
\subsection{Uniqueness of the almost compact involution} 
 \begin{proposition}  
An almost compact involution  is uniquely determined up to a sign by its action on $\g_0$, provided that the $\g_0$-module $\g_{1/2}$ is irreducible.\end{proposition}
\begin{proof} 
If there are two different extensions of the compact involution, then their ratio $\psi$, say,  is 
identical on $\g_0$, hence, by Schur's lemma, $\psi$ acts as a scalar on 
$\g_{-1/2}$. Since $\phi(f)=f$, we conclude that this scalar is $\pm 1$.
\end{proof}
It remains to discuss the cases $\g=sl(2|m),\,m\geq 3$, and  $psl(2|2)$, since in all other cases of Table 1 the $\g_0$-module $\g_{1/2}$ is irreducible. In this cases $\g$ is of type I, that is $\g_{\bar 1}=\g_{\bar 1}^+\oplus \g_{\bar 1}^-$ where  $\g_{\bar 1}^{\pm}$ are contragredient 
irreducible $\g_{\bar 0}$-modules and $[\g_{\bar 1}^{\pm},\g_{\bar 1}^{\pm}]=0$.
Let $\d_\l$ be the linear  map on $\g$ defined by setting 
\begin{equation}\label{deltal}
{\d_\l}_{|\g_{\bar 0}}=Id,\quad {\d_\l}_{|\g_{\bar 1}^+}=\l Id,\quad{\d_\l}_{|\g_{\bar 1}^-}=\l^{-1}Id.\end{equation} 
Then $\d_\l$ is an automorphism of $\g$ for any $\l\in\C$.   Suppose that $\phi'$ is another conjugate almost compact linear involution such that $\phi'_{|\g_{\bar 0}}=\phi$. Then
$\phi'= \phi\circ \gamma$ with $\gamma$ an automorphism of $\phi$ such that $\gamma_{|\g_{\bar 0}}=Id.$ If $\g=sl(2|m)$, by \cite[Lemmas 1 and 2]{SERG}, we have  $\gamma=\d_\l$.
Since $\phi(\g_{\bar 1}^+)=\g_{\bar 1}^-$ and $(\phi')^2=Id$ we have that  $\l\in\R$.
If $\g=psl(2|2)$, then $\gamma$ belongs to a three-parameter family of  automorphisms explicitly described in  \cite[\S 4.6]{GKMP}, and contained in $SL(2,\C)$. This $SL(2,\C)$ is the group of automorphisms of $\g$ corresponding to the Lie algebra $sl_2$ of outer derivations of $\g$.
\par

\begin{remark} Note that if $\phi$ is an almost compact involution, then 
$$\tilde \phi(a)=(-1)^{2j}\phi(a),\ a\in\g_{j}$$
is again an almost compact involution.
\end{remark}
\section{The bilinear form $\langle \cdot ,\cdot \rangle$ on $\g_{-1/2}$}
Consider  the following symmetric bilinear  form on $\g_{-1/2}$ 
\begin{equation}\label{sesq}\langle u,v\rangle=(e|[u,v]).\end{equation}
We want to prove the following 
\begin{proposition}\label{fh} We can choose an almost  compact involution such that $\langle \cdot ,\cdot\rangle$ is positive definite on $\mathfrak r=\{u\in\g_{-1/2}\mid  \phi(u)=-u\}.$
\end{proposition}
 The proof requires a detailed analysis of the action of an almost compact involution on $\g_{-1/2}$. Define structure constants  $N_{\a,\beta}$  for a good choice of  root vectors (see Subsection \ref{222}) by the relation $$[X_\a,X_\beta]=N_{\a,\beta}X_{\a+\beta}.$$
Observe that $\{X_{-\theta}, X_{\theta}, \tfrac{1}{2}h_\theta\}$ is a $sl_2$-triple in $\mathfrak s$. Let 
$$\g=\C X_{\theta} \oplus \tilde \g_{1/2} \oplus \tilde \g_0 \oplus \tilde \g_{-1/2} \oplus \C X_{-\theta}$$
be the decomposition into $ad\,\tfrac{1}{2}h_\theta$ eigenspaces. By $sl_2$ theory, $ad\,X_{\pm \theta}: \tilde\g_{\mp1/2}\to \tilde\g_{\pm 1/2}$ is an isomorphism 
of $\g^\natural$-modules. Moreover,  by our choice of $R^+$, the roots of $\tilde \g_{-1/2}$ (resp.  $\tilde \g_{1/2}$)  are precisely the negative (resp. positive) odd roots.  In particular, the map $\a\mapsto -\theta+\a$ defines a bijection between the positive and negative odd roots.
We shall need the following properties.
\begin{lemma} For a positive odd root $\a$ we have
\begin{align}
\label{rss} &N_{-\theta,\a}N_{\th,\a-\th}=1,\\
&\label{Npm1} N_{-\theta,\a}^2=1.
\end{align}
In particular $N_{\theta,\a}$ is real. \end{lemma}
\begin{proof}
Relation \eqref{rss} is proven in \cite[Lemma 4.3]{GKMP}. Equation \eqref{Npm1} follows from \cite[(4.8)]{GKMP}, noting that the $-\theta$-string through $\a$ has length $1$.
 \end{proof} 

Arguing as in  Proposition \ref{constructphi}, we can assume in the proof of Proposition \ref{fh} that   $\{e,x,f\}$ is the $sl_2$-triple defined in \eqref{efx}; $ad\, x$ defines on $\g$ a  minimal grading
\begin{equation}\label{mg}
\g=\C f\oplus\g_{-1/2}\oplus \g_0\oplus \g_{1/2}\oplus\C e. 
\end{equation}
Set, for an odd root $\a\in R^+$ 
\begin{equation}\label{u}u_\a=X_\a+\sqrt{-1} N_{-\theta,\a}X_{\a-\theta}.\end{equation}
Note that 
$$[x,u_\a]=\tfrac{\sqrt{-1}}{2}[X_\theta-X_{-\theta},X_\a+\sqrt{-1} N_{-\theta,\a}X_{\a-\theta}]
=\!\!\tfrac{1}{2}N_{-\theta,\a}N_{\theta,\a-\theta}X_{\a}-\tfrac{\sqrt{-1}}{2}N_{-\theta,\a}X_{\a-\theta}=-\tfrac{1}{2}u_\a,
$$
hence 
$\{u_\a\mid \a\in R^+, \a \text{ odd}\}$ is a basis of $\g_{-1/2}$. 
\begin{lemma}
 If $\a$ is a positive odd root then 
\begin{equation}\label{secondaa}\phi(u_\a)= N_{-\theta,\a} u_{\theta-\a}.\end{equation}
\end{lemma}
\begin{proof} 
By \eqref{ox}, $\phi(X_\a)=\I X_{-\a}$ if $\a$ is an odd root, hence, by \eqref{Npm1}, since $N_{-\theta,\a}$ is real,
\begin{align}\label{primaa}\begin{split}\phi(u_\a)&=\phi(X_\a+\sqrt{-1}N_{-\theta,\a}X_{\a-\theta})= 
\I X_{-\a}+N_{-\theta,\a}X_{\theta-\a}\\&=
N_{-\theta,\a}(X_{\theta-\a}+\I N_{-\theta,\a}^{-1}X_{-\a}).
\end{split}\end{align}
Note that, since $\phi(x)=x$,  $\phi(u_\a)$ has to belong to $\g_{-1/2}$. This forces 
\begin{equation} \label{ns}
N_{-\theta,\a}N_{-\theta,\theta-\a}=1,
\end{equation} 
and    \eqref{primaa} becomes \eqref{secondaa}.\end{proof}
\begin{proof}[Proof of Proposition \ref{fh}]
Set $v_\a=\tfrac{1}{2}(u_\a-\phi(u_\a))+ 	\tfrac{\I}{2}(u_\alpha+\phi(u_\alpha))$, where $\a$ runs over the positive odd roots. It is clear that $v_\a\in\mathfrak r$. We want to prove that the vectors $v_\a$ form an orthogonal basis of $\mathfrak r$. We need two auxiliary computations: 
\begin{align}\label{a1}
[e,u_\a]&=\I X_\a+N_{-\theta,\a}X_{\a-\th},\\
\label{a2} \langle u_\a,u_\be\rangle&=-(N_{-\theta,\a}+N_{-\theta,\be})\d_{\th-\a,\be}.
\end{align}
To prove \eqref{a1} use \eqref{rss}:
\begin{align*}
[e,u_\a]&=\tfrac{1}{2}[X_\th+X_{-\th}+\sqrt{-1}h_\theta,X_\a+\I N_{-\theta,\a}X_{\a-\th}]\\
&=\tfrac{1}{2}[X_\th+X_{-\th}+\sqrt{-1}h_\theta,   X_\a+\I N_{-\theta,\a}X_{\a-\th}]\\
&=\tfrac{1}{2}(\I N_{-\theta,\a}N_{\th,\a-\th}X_\a+N_{-\theta,\a}X_{\a-\th}+\I X_\a+N_{-\theta,\a}X_{\a-\th})\\
&=\I X_\a+N_{-\theta,\a}X_{\a-\th}.
\end{align*}
To prove \eqref{a2} use  \eqref{a1}
\begin{align*}\langle u_\a,u_\be\rangle&=(e|[u_\a,u_\be])=([e,u_\a]|u_\be)\\
&=(\I X_\a+N_{-\theta,\a}X_{\a-\th}|X_\be+\sqrt{-1}N_{-\theta,\be}X_{\be-\theta})=\\
&=\s_{\a-\theta}N_{-\theta,\a}\d_{\th-\a,\be}-\s_\a N_{-\theta,\be}\d_{\th-\a,\be}\\
&=-(N_{-\theta,\a}+N_{-\theta,\be})\d_{\th-\a,\be}.
\end{align*}
Set $$M_{\a,\be}=-(N_{-\theta,\a}+N_{-\theta,\be}).$$
Then, using \eqref{a2}
\begin{align*}
&\langle v_\a,v_\be\rangle\\&=\langle \tfrac{1+\I}{2} u_\a-\tfrac{1-\I}{2}N_{-\theta,\a}u_{\theta-\a}, \tfrac{1+\I}{2} u_\be-\tfrac{1-\I}{2}N_{-\theta,\be}u_{\theta-\be}\rangle\\
&=\tfrac{\I}{2}\langle u_\a, u_\be\rangle
-\tfrac{1}{2}N_{-\theta,\a}\langle u_{\theta-\a}, u_\be\rangle-\tfrac{1}{2}N_{-\theta,\be}\langle u_\a, u_{\th-\be}\rangle-\tfrac{\I}{2}N_{-\theta,\a}N_{-\theta,\be}
\langle u_{\theta-\a}, u_{\theta-\be}\rangle\\
&=
\tfrac{\I}{2}M_{\a,\be}\d_{\th-\a,\be}
-\tfrac{1}{2}N_{-\theta,\a}M_{\th-\a,\be}\d_{\a,\be}-\tfrac{1}{2}N_{-\theta,\be}M_{\a,\th-\be}\d_{\th-\a,\th-\be}\\&-\tfrac{\I}{2}N_{-\theta,\a}N_{-\theta,\be}
M_{\th-\a,\th-\be}\d_{\a,\th-\be}.\\
\end{align*}
Therefore by \eqref{rss} and \eqref{ns}
$$\langle v_\a,v_\be\rangle=2\d_{\a,\beta}.$$
In particular, the restriction of $\langle\cdot,\cdot\rangle$ to $\mathfrak r$ is positive definite.
\end{proof}

\section{Unitarity}
Let $\{e,x,f\}$ be the $sl_2$ triple constructed in Proposition \ref{constructphi}. Let $\h^\natural\subset\g^\natural$ be a Cartan subalgebra of $\g^\natural $ and fix
$$\h=\C x\oplus \h^\natural$$
as a Cartan subalgebra of $\g$. Let $\D$ be the set  of roots of $\g$ with respect to $\h$. Let $\{\pm\theta\}$ be the set of roots of $\mathfrak s$. Fix a set $\Dp$ of positive roots on $\D$ in such a way that  $\theta\in\Dp$ and $e\in \g_\theta$ (hence $f\in\g_{-\theta}$). Note that $x=\tfrac{1}{2} h_\theta$. 
 Note that $(\cdot|\cdot)_{|\mathfrak s}$ is non-degenerate, hence $(x|x)\ne 0$ and therefore we have the orthogonal direct sum of reductive subalgebras
 \begin{equation}\label{ods}
 \g_0=\C x\oplus \g^\natural.
 \end{equation}
Let $\g^\natural=\bigoplus_{i\ge 0}\g^\natural_i$ be  the decomposition of $\g^\natural$ into the direct sum of  ideals, where $\g^\natural_i$ is simple  for $i>0$ and  $\g_0^\natural$  is the center. Let $h^\vee$ be the dual Coxeter number of $\g$, and denote by  $\bar h^\vee_i$  half of the eigenvalue of the Casimir element of $\g^\natural _i$ with respect to $(\cdot |\cdot)_{|\g^\natural_i\times\g^\natural_i}$, when acting on $\g^\natural _i$. Note that $\bar h_0^\vee=0$. \par
In  \cite{KRW} the authors introduced (as a special case of a more general construction) the universal {\sl minimal} $W$-algebra $W^k_{\min}(\g)$, whose simple quotient is $W_k^{\min}(\g)$, attached to the grading \eqref{mg}. This is a vertex algebra  strongly and freely generated by elements $L$, $J^{\{v\}}$ where $v$ runs over a basis of $\g^\natural$, $G^{\{u\}}$ where $u$ runs over a basis of $\g_{-1/2}$, with the following  $\lambda$-brackets: $L$ is a Virasoro element  (conformal vector) with central charge $c(k)$ given by \eqref{cc},
 $J^{\{u\}}$ are primary of conformal weight $1$, $G^{v}$ are primary of conformal weight $\frac{3}{2}$ and (\cite[Theorem 5.1]{KW1})
\begin{align*}
[{J^{\{v\}}}_\l G^{\{u\}}]&=G^{\{[v,u]\}} &&\text{ for  $u\in \g_{-1/2}, v\in \g^\natural$,}\\
[{J^{\{v\}}}_\l J^{\{w\}}]&=J^{\{[v,w]\}}+\l\beta_k(v,w) &&\text{ for $v,w\in \g^\natural$.}
\end{align*}
Here
\begin{equation}
\label{eq:5.17}\beta_k (u,v) =\d_{i,j}(k+ \tfrac{h^\vee-\bar h_i^\vee}{2})(u|v),\quad u\in\g^\natural_i, v\in\g^\natural_j.
\end{equation}
The most explicit formula for the $\l$-bracket between the $G^{\{u\}}$ is given in \cite[Proposition 5.8]{Y}. It turns out that a crucial role is played by a certain monic quadratic polynomial $p(k)$, introduced in \cite[Table 4]{AKMPP} and thoroughly investigated in \cite{Y}. (In Remark \ref{63} we define it explicitly.) The following relation will be relevant in the sequel (see \cite[Theorem 3.2]{AKMPP}):
\begin{equation}\label{GG}
{G^{\{u\}}}_{(2)} G^{\{v\}}= 4 (e_\theta| [u,v])p(k).\end{equation}
\vskip5pt
The following proposition is a special case of \cite[Lemma 7.3]{KMP}, in view of Lemma \ref{31}.
\begin{prop}\label{descend}
Let $\phi$ be a conjugate linear involution of $\g$ such that
$
\phi(f)=f, \phi(x)=x, \phi(e)=e.
$
Then the map
\begin{equation}\label{phiespilcita} \phi(J^{\{u\}})= J^{\{\phi(u)\}},\quad \phi(G^{\{v\}})=G^{\{\phi(v)\}},\quad \phi(L)=L,\ \ u\in \g^\natural,\ v\in \g_{-1/2}\end{equation}
 extends to
 a  conjugate linear involution of the vertex algebra $W^k_{\min}(\g)$.
\end{prop}

By Proposition \ref{constructphi}  there is a    conjugate linear involution $ \phi$ on $\g$ such that $\phi(x)=x,\,\phi(f)=f$ and $(\g^\natural)^{\phi}$ is a compact real form of $\g^\natural$, hence, by Proposition \ref{descend}, $\phi$ induces a  conjugate linear involution of the vertex algebra $W^k_{\min}(\g)$, and descends to  a  conjugate linear involution $\phi$ of its unique simple quotient $W_k^{\min}(\g)$, which we again denote by $\phi$.

  By \cite[Proposition 7.4 (b)]{KMP}, 
$W^k_{\min}(\g)$ admits a unique $\phi$-invariant Hermitian form $H(\cdot,\cdot)$ such that $H(\vac,\vac)=1$. Recall that if
$k+h^\vee\ne 0$ then the kernel of $H(\cdot,\cdot)$  is the unique maximal ideal of  $W^k_{\min}(\g)$, hence $H(\cdot,\cdot)$  descends to define a $\phi$-invariant Hermitian form on $W_k^{\min}(\g)$, which we again denote by $H(\cdot,\cdot)$. 
\vskip5pt
In the following we investigate the unitarity of  $W_k^{\min}(\g)$ (cf. Definition \ref{unidef}) when $\g$ is as in Section \ref{2}. With a slight abuse of terminology, we also say that 
$W^k_{\min}(\g)$ is  unitary if  $H(\cdot,\cdot)$ is positive semidefinite. It is clear that for   $k\ne - h^\vee$,  $W^k_{\min}(\g)$ is  unitary if and only if $W_k^{\min}(\g)$ is unitary.

We need  to fix notation  for affine vertex algebras.
Let $\aa$ be a Lie superalgebra equipped with a nondegenerate  invariant supersymmetric bilinear form $B$. The universal affine vertex algebra $V^B(\aa)$ is  the universal enveloping vertex algebra of  the  Lie conformal superalgebra $R=(\C[T]\otimes\aa)\oplus\C$ with $\lambda$-bracket given by
$$
[a_\lambda b]=[a,b]+\lambda B(a,b),\ a,b\in\aa.
$$
In the following, we shall say that a vertex algebra $V$ is an affine vertex algebra if it is a quotient of some $V^B(\aa)$. If $\aa$ is simple Lie algebra, we denote by $(\cdot|\cdot)^\aa$ the normalized invariant bilinear form on $\aa$, defined by the condition $(\a|\a)^\aa=2$ for a long root $\a$. Then  $B=k(\cdot|\cdot)^\aa$, and  we simply write $V^k(\aa)$.  If $k\ne -h^\vee$, then  $V^k(\aa)$ has a unique simple quotient, which will be denoted by 
$V_k(\aa)$.

Consider the universal affine vertex algebra $V^{\a_k}(\g_0)$ with $\lambda$-bracket (see \cite[(5.16)]{KW1})
$$[a_\lambda b]=[a,b]+\l\a_k(a,b),$$
where
\begin{equation}
\label{eq:5.16}
  \alpha_k (a, b) = ((k+h^\vee)(a|b)-\tfrac{1}{2}
\kappa_{\fg_0}(a,b))\,,
\end{equation}
and where $\kappa_{\fg_0}$ denotes the Killing form of $\fg_0$. Note that $$\a_k(a,b)= \d_{i,j}(k+h^\vee-\bar h_i^\vee)(a|b)\text{ if } a\in\g^\natural_i,\ b\in \g^\natural_j.$$
\par
Let $\psi$ be a conjugate linear involution of $\g$ such that $
(\psi(x) |\psi(y))=\overline{(x|y)}$. By \cite[\S5.3]{KMP} there exists a unique $\psi$-invariant Hermitian form $H$ on $V^k(\g)$.  If $k\ne - h^\vee$, then the kernel of  $H$ is the maximal ideal of $V^k(\g)$, hence $H$ can be pushed down to $V_k(\g)$.\par
Recall that  we assume that $\g^\natural$ is not abelian. Let $\theta_i$ be the highest root of $\g^\natural_i$ for $i>0$. Set 
$$z_i(k)=
\frac{2}{u_i}\left(k+\frac{h^\vee-\bar h^\vee_i}{2}\right),$$
where
$$u_i=\begin{cases}2\quad&\text{if $i=0$,}\\
(\theta_i|\theta_i)\quad&\text{if $i>0$.}
\end{cases}
$$
Let  $(\cdot|\cdot)_i^\natural$ denote the normalized invariant bilinear form on $\g^\natural_i$ for $i>0$ and let 
$(\cdot|\cdot)_0^\natural=(\cdot|\cdot)_{|\g_0^\natural\times\g_0^\natural}$.  Note that, for $i>0$,  $(a|b)_i^\natural= \d_{i,j}\frac{(\theta_i|\theta_i)}{2}(a|b)$,  hence,\begin{equation}
\label{ab}
\beta_k(a, b)=\d_{i,j} z_i(k)(a|b)_i^\natural  \text{ for $a\in \g^\natural_i,\ b\in \g^\natural_j$.}
\end{equation}
Moreover we have
\begin{equation}
\label{abb}
\a_k(a, b)=\d_{i,j} \frac{2}{(\theta_i|\theta_i)}\left(k+h^\vee-\bar h^\vee_i\right)(a|b)^\natural= \d_{i,j} (z_i(k)+\chi_i)(a|b)^\natural \text{ for $a\in \g^\natural_i, b\in \g^\natural_j$,}
\end{equation}
where 
\begin{equation}\label{chii}\chi_i=\frac{h^\vee-\bar h^\vee_i}{u_i},\ i\geq 0.\end{equation}
The relevant data for computing $\chi_i,$ for  $i\geq 0$ are collected in  Table 2, where the explicit value of $z_i(k)$ for $i\geq 0$ is also displayed. 
Note that $z_0(k)=k+\tfrac{1}{2}h^\vee$.
\begin{table}{\small
\begin{tabular}{   c | c| c| c| c | c | c}
$\g$ & $\g^\natural$ & $u_i$ & $h^\vee$ & $\bar h_i^\vee$ & $z_i(k)$ &$\chi_i$\\\hline
$sl(2|m), m\ge 3$ & $ \C\oplus sl_m$ & $2,-2$ & $2-m$ & $0,-m$ & $k-\frac{m-2}{2}, -k-1$&$1-m/2,-1$\\\hline
$psl(2|2)$& $sl_2$  & $-2$ &  $0$ & $-2$ & $-k-1$&$-1$\\\hline
$osp(4|m), m>2$ & $sl_2\oplus sp_m$&$2, -4$ & $2-m$ & $ 2, -m-2$ & $k-\frac{m}{2}, -\frac{1}{2}k-1$ &$-m/2,-1$\\\hline
$spo(2|3)$& $sl_2$  &$-1/2$ & $1/2$ & $-1/2$ & $-4k-2$&$-2$\\\hline
$spo(2|m), m>3$ & $so_m$ &  $-1$& $2-m/2$ & $1-m/2$ &  $-2k-1$&$-1$\\\hline
$D(2,1;a)$ &  $sl_2\oplus sl_2$ & $-2-2a, 2a$&$0$ &  $-2-2a,2a $ & $-\frac{1}{1+a}k-1, \frac{1}{a}k-1$ &$-1,-1$\\\hline
$F(4)$& $so_7$  & $-4/3$ &$-2$ & $-10/3$ & $-\frac{3}{2} k-1$ &$-1$\\\hline
$G(3)$& $G_2$ &$-2/3$  &$-3/2$  &  $-3$ &$ -\frac{4}{3} k-1$&$-1$
\end{tabular}
}
\vskip5pt
\caption{}
\end{table}

Recall from \cite{AKMPP} that a level $k$ is {\it collapsing} for $W_k^{\min}(\g)$ if $W_k^{\min}(\g)$ is either $\C$ or  the  simple affine vertex algebra over $\g^\natural$.  The main result of \cite{AKMPP} states that $k$ is collapsing if and only if $p(k)=0$.

We summarize in the following result the content  of Theorem 3.3 and Proposition 3.4 of \cite{AKMPP} relevant to our setting.
\begin{theorem}\label{oldresults} Let $\g$ be a basic simple Lie superalgebra from  Table 2. Assume $k\ne -h^\vee$.
\begin{enumerate}
\item If $\g^\natural$ is simple then $k$ is collapsing if and only if  $z_1(k)=0$ or $k=-\tfrac{\bar h_{1}^\vee}{2}-1$.
\item If $\g^\natural$ is not simple then $k$ is collapsing if and only if there is $i$ such that   $z_i(k)=0$.
\end{enumerate} 
Moreover we have $W_k^{\min}(\g)=\C$ if and only if $\g^\natural$ is simple and $z_1(k)=0$.
\end{theorem}
\begin{remark}\label{63} We say that an ideal in $\g^\natural$ is a {\it component} of $\g^\natural$ if it is simple or $1$-dimensional.
It follows from \cite[Lemma 3.1]{AKMPP},  \cite[Theorem 5.9]{KMP} that, up to a constant factor 
$$p(k)=\begin{cases} z_1(k)z_2(k)\quad & \text{if $\g^\natural$ has two components,}\\
z_1(k)(k+\frac{\bar h_1^\vee}{2}+1)\quad & \text{otherwise.}
\end{cases}$$
The roots of $p(k)$ are the collapsing levels defined in the Introduction.
\end{remark}

\begin{theorem}\label{main} Let $\g$ be a Lie superalgebra from  Table 2 with $\g^\natural$ non abelian    and let $W_k^{\min}(\g)$ be the corresponding  minimal simple $W$-algebra. Assume that $k\ne-h^\vee$. 
Then

\noindent (a)  $W_k^{\min}(\g)$ is unitary precisely in the following cases:
\begin{enumerate}
\item $\g^\natural$ is semisimple and $z_i(k)\in \mathbb Z_+$ for all $i>0$. 
\item $\g=sl(2|m), m\geq 3,\  k=-1$.
\end{enumerate}

\noindent (b) 
If there exists a unitary module for $W_k^{\min}(\g)$, then $W_k^{\min}(\g)$ is unitary. 
\end{theorem}
\begin{proof} Assume first that $\g^\natural$ is semisimple. 
We first prove that, if $z_i(k)\in \mathbb Z_+$ for all $i>0$ then  $W_k^{min}(\g)$ is unitary.

Recall from \cite[Theorem 5.2]{KW1} that there is a vertex algebra homomorphism $\Psi:W^k_{\min}(\g)\to V^{\a_k}(\g_0)\otimes F(A_{ne})$. We equip $W^k_{\min}(\g)$ and  $V^{\a_k}(\g_0)\otimes F(A_{ne})$ with their invariant Hermitian forms. Using the explicit description of the embedding given in \cite{KW1}, by \eqref{phiespilcita}, one can check that $\Psi$ preserves the forms.  

If  $\aa$  is a quadratic abelian subalgebra of $\g_{\bar 0}$, and $B=t(\cdot | \cdot)_{|\aa\times \aa}$ we denote $V^B(\aa)$ by $V^t(\aa)$.
By \eqref{abb}
$$V^{\a_k}(\g_0)=\left(\bigotimes_{i\ge 1} V^{z_i(k)+\chi_i}(\g_i^\natural)\right)\otimes V^{k+h^\vee}(\C x).$$
Since $\phi$ from Proposition \ref{constructphi} corresponds to a compact real form of $\g^\natural$, if $z_i(k)+\chi_i\in \mathbb Z_+$, the $\phi$--invariant Hermitian form on  $V^{z_i(k)+\chi_i}(\g^\natural_i)$ is positive semidefinite.  \par Recall from \cite[\S 5.2]{KMP} that if $\aa$ is an even abelian 
Lie algebra with a symmetric bilinear form $B$ and $\eta$ is a conjugate linear involution such that $B(\eta(a), \eta(b))=\overline{B(a,b)},\,a,b\in\aa$, then letting $\eta$ be the unique extension of $\eta$ to $V^B(\aa)$, the unique $\eta$-invariant Hermitian form on  $V^B(\aa)$ is positive definite if and only if $B_{|\aa_\R\times  \aa_\R}$ is positive definite, where $\aa_\R=\{a\in\aa\mid \eta(a)=-a\}$. \par 
Applying this remark to 
$\aa=\C x, \eta=\phi$ and $B=k+h^\vee(\cdot |\cdot)$, since
$\phi(x)=x$ and $(x|x)=\frac{1}{4}(\theta|	
\theta)=\frac{1}{2}$, the $\phi$--invariant Hermitian form on $V^{k+h^\vee}(\C x)$ is positive definite if and only if 
$$
(k+h^\vee)(x|x)<0\Leftrightarrow k+h^\vee<0.
$$
Finally we check the unitarity of $F(A_{ne})=F(\g_{1/2})$. We have to check that, according to \cite[\S 5.1]{KMP},  the restriction of $\langle\cdot, \cdot\rangle$ to $\{u\in\g_{-1/2}\mid \phi(u)=-u\}$  is positive definite.  This is  Proposition \ref{fh}.
 
It follows that, if $z_i(k)+\chi_i\in \mathbb Z_+$ and $k+h^\vee<0$, then $V^{\a_k}(\g_0)\otimes F(A_{ne})$ is unitary, hence $W_k^{\min}(\g)$ is. By looking at Table 2 one sees that $\chi_i\in-\ZZ_+$ and checks that $z_i(k)\geq 0$ implies $k+h^\vee<0$. Thus, if $z_i(k)\ge -\chi_i$, then $W_k^{\min}(\g)$ is unitary.

 Now we look at the missing cases $0\le z_i(k)< -\chi_i$. Assume first $\g^\natural$ simple. If $\chi_1=-1$  then the only value not covered by the above condition is $z_1(k)=0$,
 so, by Theorem \ref{oldresults},  $W_k^{\min}(\g)=\C$.  In the case of $spo(2|3)$ one should also consider the cases $z_1(k)=1, z_1(k)=0$: by Theorem \ref{oldresults} in the former case $k$ is collapsing with $W_k^{\min}(spo(2|3))=V_1(sl_2)$, whereas in the latter case $k=-h^\vee$. If $\g^\natural$ is semisimple but not simple, then $\g$ is either $osp(4|m)$ or $D(2,1;a)$.
For $\g=D(2,1;a)$ we have to consider only the case in which either $z_1(k)$ or $z_2(k)$ is zero. By Theorem \ref{oldresults}, they correspond to collapsing levels. For $\g=osp(4|m)$, our condition reduces to
$$m/2\le k<m,\ k\le -2,\text{ or }m/2\le k,\ -4< k\le -2
$$
which is never satisfied.

We now prove that our condition is also necessary, if $\g^\natural$ is semisimple. Assume that there is a conjugate linear involution $\psi$ of $W_k^{\min}(\g)$ such that the $\psi$--invariant Hermitian form is positive definite.
Recall  that the map $a\mapsto J^{\{a\}}$ extends to  an embedding $V^{\beta(k)}(\g^\natural)\subset W^k_{\min}(\g)$. In particular the image of $V^{\beta(k)}(\g^\natural)$ in $W_k^{\min}(\g)$ is unitary and simple.

By \eqref{ab},
$$ V^{\beta(k)}(\g^\natural)=\bigotimes_{i\ge 1} V^{z_i(k)}(\g_i^\natural),$$
hence we have an embedding 
$$
\bigotimes_{z_i(k)\ne 0} V_{z_i(k)}(\g_i^\natural)\subset W_k^{\min}(\g).
$$
By \S 5.3 of \cite{KMP}, we see that $\psi$ restricted to $\bigoplus\limits_{z_i(k)\ne0}\g_i^\natural$ defines a compact real form and, if $z_i(k)\ne0$,  then $z_i(k)\in\nat$. Therefore $z_i(k)\in\ZZ_+$ for all $i$.

\par We now discuss the case when $\g$ is not semisimple. Note that  $\g_0^\natural\ne \{0\}$ only when $\g=sl(2|m)$. In this case $\g_0^\natural= \C \varpi$, where $\varpi=\left(
\begin{array}{cc|c}
m/2 & 0 &  0\\0 & -m/2   & 0\\   \hline 0& 0 & I_m
\end{array}\right)$, and $(a|b)=str(ab)$. By Theorem \ref{oldresults}, the  collapsing levels are $k=-1$ and $k=-m/2-1$.
If $k=-1$ then $z_0(-1)=-m/2,  z_1(-1)=0$ and  $W_k^{\min}(\g)$ is the Heisenberg vertex algebra $M(\C\varpi)=V^{-m/2}(\C\varpi)=V_{-m/2}(\C\varpi)$ and this vertex algebra is unitary.  If $k=-m/2-1$ then  $z_0(m/2-1)=0,  z_1(m/2-1)=-m/2$ and $W_k^{\min}(sl(2|m))=V_{-m/2}(sl(m))$ which is not unitary.
Assume now that $k$ is not collapsing. Let $\psi$ be a conjugate linear involution of $W_k^{\min}(sl(2|m))$ such that the $\psi$-invariant Hermitian form 
$H$ is positive definite. In particular there is an embedding 
$$V^{k-m/2-1}(\C\varpi)\otimes V_{-k-1}(sl(m))\hookrightarrow W_k^{\min}(sl(2|m))$$
Hence $\psi$ induces by restriction a conjugate linear involution of $\C\varpi\oplus sl(m)$, thus $\psi(\varpi)=\zeta\varpi,  |\zeta|=1$. 
Moreover, $\psi_{|sl(m)}$ corresponds to a compact real of form of $sl(m)$ and $-k-1\in\Z_+$.
Let $\langle b \rangle_{ev}$ denote the expectation value of $b$, i.e.
the coefficient of the projection of $b$ on the vacuum vector $\vac$. Using the formulas given in
\cite[\S 5.3]{KMP} we have 
$$0<H(J^{\{\varpi\}},J^{\{\varpi\}})\!=\!\langle -J^{\{\psi(\varpi)\}}_1J^{\{\varpi\}}_{-1}\rangle_{ev}=-(k-m/2-1)\zeta (\varpi|\varpi)=-\zeta(k-m/2-1)(m^2/2-m).
$$
Therefore $\zeta=1$, so that 
\begin{equation}\label{=pi}\psi(\varpi)=\varpi.
\end{equation}
Observe now that by \cite{AKMPP}, since $k$ is not collapsing, the image of $G^{\{u\}}$ in $W_k^{\min}(\g)$ is non-zero if $u\ne 0$. Note that, from the relation $[J^{\{a\}},G^{\{u\}}]=J^{\{[a,u]\}}$, $a\in\g_0, u\in \g_{-1/2}$,  it follows that $\psi$ can be extended to $\g_{-1/2}$, in such a way that $\psi([a,u])=[\psi(a),\psi(u)]$.\par
Note that 
\begin{equation}\label{id} [\varpi, u]=-\tfrac{m}{2}u,\quad u\in\g_{-1/2}.
\end{equation}
Take $u\in\g_{-1/2}$. Compute, using \eqref{id} and \eqref{=pi}
\begin{align*}
\langle\psi(u)|u\rangle&=-\tfrac{2}{m}\langle\psi(u)|[\varpi,u]\rangle=\tfrac{2}{m}\langle[\varpi,\psi(u)]|u\rangle\\&=\tfrac{2}{m}\langle[\psi(\varpi),\psi(u)]|u\rangle
=\tfrac{2}{m}\langle\psi([\varpi,u])|u\rangle=-\langle\psi(u)|u\rangle,
\end{align*}
so $\langle\psi(u)|u\rangle=0$. But this contradicts positivity of $H$, since, by \cite[(7.6)]{KMP} and \eqref{GG}
\begin{align*}H(G^{\{u\}},G^{\{u\}})&=H(G^{\{u\}}_{-3/2}\vac,G^{\{u\}}_{-3/2}\vac)=\left\langle G^{\{\psi(u)\}}_{3/2}G^{\{u\}}_{-3/2}\vac\right\rangle_{ev}=\left\langle {G^{\{\psi(u)\}}}_{(2)}G^{\{u\}}\right\rangle_{ev}\\&=4p(k)\langle \psi(u)|u\rangle.
\end{align*}

For (b) observe that, if there is a unitary module for $W^k_{\min}(\g)$ then also $V^{z_i(k)}(\g^\natural)$ admits a unitary module, but then $z_i(k)\in \ZZ_{\ge0}$.
\end{proof}
\begin{theorem}\label{fin} Let $\g$ be a basic simple Lie superalgebra from Table 1.   Let $V=W_{-k}^{\min}(\g)$ with  
$k\neq h^\vee$. Assume that $\g\ne spo(2|m), m=0,1,2$, which correspond to the well-understood cases of Virasoro, Neveu-Schwarz and $N=2$ superconformal algebra. Then  $V$ is a non-trivial unitary vertex algebra if and only if
\begin{enumerate}
\item $\g=sl(2|m), m\geq 3,\  k=1$;
\item  $\g=psl(2|2)$, $k\in\mathbb N+1$;
\item  $\g=spo(2|3)$, $k\in\frac{1}{4}(\mathbb N+2)$;
\item $\g=spo(2|m), m>3$, $k=\frac{1}{2}(\mathbb N+1)$;
\item   $\g=D(2,1;-\frac{m}{m+n})$, $k=\frac{mn}{m+n},\ m,n\in\mathbb N, m+n>2$;
\item   $\g=F(4)$, $k\in\frac{2}{3}(\mathbb N+1)$;
\item  $\g=G(3)$, $k\in\frac{3}{4}(\mathbb N+1)$.
\end{enumerate}
\end{theorem}
\begin{proof} The result is obtained by a direct application of Theorem \ref{main}, using the data displayed in  Table 2; one should take care, using  Theorem \ref{oldresults}, to exclude (collapsing)  levels for which $V$ is trivial. Combining the conditions coming from Theorems  \ref{main} and \ref{oldresults}, we obtain that   $z_1(-k)$ has to be a positive integer if $\g^\natural$  is simple: this covers cases (2), (3), (4), (6), (7). For the remaining cases of Table 1, $sl(2|m)$ is handled directly by Theorem  \ref{main}, (2). If $\g=osp(4|m)$, then the conditions $z_1(-k)\in \mathbb Z_+,  z_2(-k)\in \mathbb Z_+$ are not compatible, hence in this case $V$ is never unitary. It remains to deal with $\g=D(2,1;a)$. Then 
by Theorem \ref{main} we have
$-\frac{1}{a}k=N+1$, hence
\begin{equation}\label{primaaa} k= -a(N+1), \ N\in\mathbb Z_+\end{equation}
and  $\frac{1}{1+a}k-1=M\in \mathbb Z_+$, hence, substituting \eqref{prima}, we have
$$-\frac{a}{1+a}(N+1)-1=M,$$
hence $a(M+N+2)+M+1=0,$ and 
$$ a=-\frac{M+1}{M+N+2}=-\frac{m}{m+n},\text{ where } m=M+1,\  n=N+1\in \mathbb N$$
Substituting in \eqref{primaaa}, we get 
$$k=\frac{mn}{m+n}.$$
The case $m=n=1$ is excluded since in this case $c(k)=0$, hence $W^{\min}_k(\g)$ is trivial, since it is unitary.
\end{proof}
\begin{remark} Replacing $\mathbb N$ by $0$ in cases (2)-(4), (6) and (7) of Theorem \ref{fin} we obtain $k$, for which $W_{-k}^{\min}(\g)=\C$.
In case (5) this happens for $m=n=1$.
\end{remark}
\begin{remark} It has been proved in \cite{KMP} that if $\g$ does not appear in Table 1, then $W_k^{\min}(\g)$ can be unitary only if $k$ is collapsing. 
Then $W_k^{\min}(\g)$ is non trivial unitary precisely in the following cases.
\begin{enumerate}
\item $W^{\min}_{-1}(sl(m|n))\cong M(\C),\,m\neq n,n+1,n+2, m\ge2,$ where $M(\C)$ is the Heisenberg vertex algebra  with central charge $c=1$;
\item $W^{\min}_{4/3}(G_2)\cong V_1(sl(2))$ with central charge $c=1$;
\item $W^{\min}_{-2}(osp(m|n))\cong  V_{\frac{m-n-8}{2}}(sl(2)), m-n\geq 10, m\text{ and } n\text{ even}$,  with central charge  \newline $c=\frac{3(m-n-8)}{m-n-4}$.
\end{enumerate}
\begin{remark} Unitary representations for the $N=3$ and $N=4$ superconformal algebras have been studied in \cite{M}, and \cite{ET1}, 
\cite{ET2}, \cite{ET3}, respectively. The same central charges as in cases (2) and (3) of Theorem
\ref{fin}  appear  in \cite[(6)]{ET3}, \cite[(2.3.1)]{M}, respectively.
\end{remark}
\end{remark}
     \begin{remark}
It follows easily from \cite[\S 6, (2.6)]{KW1} for any $W_k(\g,x,f)$ that, if $d \,h^\vee\ne 0$, there are precisely two values $k_1,k_2$ of $k$ which afford the same central charge $c(k)$, and they are related  by 
\begin{equation}\label{relate}
(k_1+h^\vee)(k_2+h^\vee)=\frac{d\,h^\vee}{12(x|x)}.
\end{equation}
This follow from \eqref{ccc} by writing $c(k_1)-c(k_2)=0$. 
A direct check shows that, except for $\g=spo(2|m)$ with $m=0 $ or $1$, if $k_1$ is in the non-trivial unitary range, then $k_2$ is not. 
In the case of $spo(2|m)$ this can be proved using 
that the right hand side of \eqref{relate} equals
$$\frac{(2-m)(3-m)(4-m)}{24}.$$

\end{remark}

\vskip20pt
    \footnotesize{
\noindent{\bf V.K.}: Department of Mathematics, MIT, 77
Mass. Ave, Cambridge, MA 02139;\newline
{\tt kac@math.mit.edu}
\vskip5pt
\noindent{\bf P.MF.}: Politecnico di Milano, Polo regionale di Como,
Via Anzani 4, 22100, Como, Italy;\newline {\tt pierluigi.moseneder@polimi.it}
\vskip5pt
\noindent{\bf P.P.}: Dipartimento di Matematica, Sapienza Universit\`a di Roma, P.le A. Moro 2,
00185, Roma, Italy;\newline {\tt papi@mat.uniroma1.it}
}

   \end{document}